\documentclass[12pt]{amsart}
\usepackage{amscd,amsmath,amsthm,amssymb,enumerate}
\usepackage[left]{lineno}
\usepackage{pstricks}
\usepackage[utf8]{inputenc}
\usepackage{epstopdf}

\usepackage{comment}

\usepackage[all]{xy}

 \def\NZQ{\mathbb}               
 \def\NN{{\NZQ N}}
 
 \def\ZZ{{\NZQ Z}}

 \def\G{{\mathcal G}}

 \def\opn#1#2{\def#1{\operatorname{#2}}} 
 \opn\chara{char} \opn\length{\ell} \opn\pd{pd} \opn\rk{rk}
 \opn\projdim{proj\,dim} \opn\injdim{inj\,dim} \opn\rank{rank}
 \opn\depth{depth} \opn\grade{grade} \opn\height{height}
 \opn\embdim{emb\,dim} \opn\codim{codim}
 
 \opn\Tr{Tr} \opn\bigrank{big\,rank}
 \opn\superheight{superheight}\opn\lcm{lcm}
 \opn\trdeg{tr\,deg}
 \opn\reg{reg} \opn\lreg{lreg} \opn\ini{in} \opn\lpd{lpd}
 \opn\size{size} \opn\sdepth{sdepth}
 \opn\link{link}\opn\fdepth{fdepth}\opn\lex{lex}

 \opn\div{div} \opn\Div{Div} \opn\cl{cl} \opn\Cl{Cl}
 \opn\Spec{Spec} \opn\Supp{Supp} \opn\supp{supp} \opn\Sing{Sing}
 \opn\Ass{Ass} \opn\Min{Min}\opn\Mon{Mon}
 \opn\Ann{Ann} \opn\Rad{Rad} \opn\Soc{Soc}
 \opn\Im{Im} \opn\Ker{Ker} \opn\Coker{Coker} \opn\Am{Am}
 \opn\Hom{Hom} \opn\Tor{Tor} \opn\Ext{Ext} \opn\End{End}
 \opn\Aut{Aut} \opn\id{id}
 
 \opn\nat{nat}
 \opn\pff{pf}
 \opn\Pf{Pf} \opn\GL{GL} \opn\SL{SL} \opn\mod{mod} \opn\ord{ord}
 \opn\Gin{Gin} \opn\Hilb{Hilb}\opn\sort{sort}
 \opn\aff{aff} \opn
 \con{conv} \opn\relint{relint} \opn\st{st}
 \opn\lk{lk} \opn\cn{cn} \opn\core{core} \opn\vol{vol}  \opn\inp{inp} \opn\nilpot{nilpot}
 \opn\link{link} \opn\star{star}\opn\lex{lex}\opn\set{set}
 \opn\width{wd}
 \opn\ecart{ecart}
 \opn\gr{gr}
 
 \def\pot#1#2{#1[\kern-0.28ex[#2]\kern-0.28ex]}

 \opn\dirlim{\underrightarrow{\lim}}
 \opn\inivlim{\underleftarrow{\lim}}
 \let\to=\rightarrow
 
 \def\Implies{\ifmmode\Longrightarrow \else
         \unskip${}\Longrightarrow{}$\ignorespaces\fi}
 \def\implies{\ifmmode\Rightarrow \else
         \unskip${}\Rightarrow{}$\ignorespaces\fi}
 \def\iff{\ifmmode\Longleftrightarrow \else
         \unskip${}\Longleftrightarrow{}$\ignorespaces\fi}

 \let\:=\colon

 \def\Soc{{\mathbf Soc}}

 \def\opn#1#2{\def#1{\operatorname{#2}}} 
 \opn\chara{char} \opn\length{\ell} \opn\pd{pd} \opn\rk{rk}
 \opn\projdim{proj\,dim} \opn\injdim{inj\,dim} \opn\rank{rank}
 \opn\depth{depth} \opn\grade{grade} \opn\height{height}
 \opn\bigheight{bigheight}
 \opn\embdim{emb\,dim} \opn\codim{codim}
 
 \opn\superheight{superheight}\opn\lcm{lcm}
 \opn\trdeg{tr\,deg}
 \opn\reg{reg} \opn\lreg{lreg} \opn\ini{in} \opn\lpd{lpd}
 \opn\size{size} \opn\sdepth{sdepth}
 \opn\link{link}\opn\fdepth{fdepth}\opn\lex{lex}
 \opn\type{type}
 \opn\gap{gap}
 \opn\arithdeg{arith-deg}
 \opn\Deg{Deg}
 \opn\sat{sat}
 \opn\mat{mat}
 \opn\Mat{Mat}
 \opn\div{div} \opn\Div{Div} \opn\cl{cl} \opn\Cl{Cl}
 \opn\Spec{Spec} \opn\Supp{Supp} \opn\supp{supp} \opn\Sing{Sing}
 \opn\Ass{Ass} \opn\Min{Min}\opn\Mon{Mon} \opn\Max{Max}
 \opn\Ann{Ann} \opn\Rad{Rad} \opn\Soc{Soc}
 \opn\Im{Im} \opn\Ker{Ker} \opn\Coker{Coker} \opn\Am{Am}
 \opn\Hom{Hom} \opn\Tor{Tor} \opn\Ext{Ext} \opn\End{End}
 \opn\Aut{Aut} \opn\id{id}
 
 \opn\nat{nat}
 \opn\pff{pf}
 \opn\Pf{Pf} \opn\GL{GL} \opn\SL{SL} \opn\mod{mod} \opn\ord{ord}
 \opn\Gin{Gin} \opn\Hilb{Hilb}\opn\sort{sort}
 \opn\PF{PF}\opn\Ap{Ap}
 \opn\mult{mult}
 \opn\bight{bight}
 \opn\aff{aff}
 \opn\relint{relint} \opn\st{st}
 \opn\lk{lk} \opn\cn{cn} \opn\core{core} \opn\vol{vol}  \opn\inp{inp} \opn\nilpot{nilpot}
 \opn\link{link} \opn\star{star}\opn\lex{lex}\opn\set{set}
 \opn\width{wd}
 \opn\Fr{F}
 \opn\QF{QF}
 \opn\G{G}
 \opn\type{type}\opn\res{res}
 \opn\conv{conv}
 \opn\Shad{Shad}
 \opn\gr{gr}
 
 \def\pot#1#2{#1[\kern-0.28ex[#2]\kern-0.28ex]}

 \opn\dirlim{\underrightarrow{\lim}}
 \opn\inivlim{\underleftarrow{\lim}}
 \let\to=\rightarrow
 
 \def\Implies{\ifmmode\Longrightarrow \else
         \unskip${}\Longrightarrow{}$\ignorespaces\fi}
 \def\implies{\ifmmode\Rightarrow \else
         \unskip${}\Rightarrow{}$\ignorespaces\fi}
 \def\iff{\ifmmode\Longleftrightarrow \else
         \unskip${}\Longleftrightarrow{}$\ignorespaces\fi}

 \let\:=\colon
\theoremstyle{plain}
\newtheorem{theorem}{Theorem}[section]
\newtheorem{Theorem}[theorem]{Theorem}
\newtheorem{thm}[theorem]{Theorem}

\newtheorem{Proposition}[theorem]{Proposition}

\newtheorem{Corollary}[theorem]{Corollary}
\newtheorem{cor}[theorem]{Corollary}

\newtheorem{Lemma}[theorem]{Lemma}

\newtheorem{claim}{Claim}

\theoremstyle{definition}

\newtheorem{Definition}[theorem]{Definition}

\newtheorem{ex}[theorem]{Example}

\newtheorem{quest}[theorem]{Questions}

\newtheorem{rem}[theorem]{Remark}

\newtheorem{prob}[theorem]{Problem}

\newtheorem{Fact}[theorem]{Fact}
\newtheorem*{acknowledgments}{Acknowledgments}
\newtheorem*{condition}{Conditions}

 \let\epsilon\varepsilon
 \let\kappa=\varkappa
 \def\qed{\ifhmode\textqed\fi
       \ifmmode\ifinner\quad\qedsymbol\else\dispqed\fi\fi}
 \def\textqed{\unskip\nobreak\penalty50
        \hskip2em\hbox{}\nobreak\hfil\qedsymbol
        \parfillskip=0pt \finalhyphendemerits=0}
 \def\dispqed{\rlap{\qquad\qedsymbol}}
 
 \opn\dis{dis}
 \def\pnt{{\raise0.5mm\hbox{\large\bf.}}}
 
 \opn\Lex{Lex}

\newcommand{\rme}{\mathrm{e}}

\newcommand{\rmr}{\mathrm{r}}

\newcommand{\rmF}{\mathrm{F}}

\newcommand{\rmQ}{\mathrm{Q}}

\newcommand{\fkm}{\mathfrak{m}}
\newcommand{\fkn}{\mathfrak{n}}

\newcommand{\fkp}{\mathfrak{p}}

\newcommand{\fkM}{\mathfrak{M}}

\def\ol{\overline}

\def\tr{\mathrm{tr}}

\title{The tiny trace ideals of the canonical modules in Cohen-Macaulay rings of dimension one}

\author{J\"{u}rgen Herzog}
\address{J\"urgen Herzog: Fachbereich Mathematik, Universit\"at Duisburg-Essen, Fakult\"at f\"ur Mathematik, 45117 Essen, Germany}
\email{juergen.herzog@uni-essen.de}

\author{Shinya Kumashiro}
\address{Shinya Kumashiro: National Institute of Technology (KOSEN), Oyama College
771 Nakakuki, Oyama, Tochigi, 323-0806, Japan}
\email{skumashiro@oyama-ct.ac.jp}

\author{Dumitru I. Stamate}
\address{Dumitru I. Stamate: Faculty of Mathematics and computer science, University of Bucharest, Str. Academiei 14, Bucharest - 010014, Romania}
\email{dumitru.stamate@fmi.unibuc.ro}

\thanks{The second author was supported by JSPS KAKENHI Grant Number 21K13766.}

\begin{document}

\begin{abstract}
We study one-dimensional Cohen-Macaulay rings whose trace ideal of the canonical module is as small as possible. In this paper we call such rings far-flung Gorenstein rings. 
We investigate far-flung Gorenstein rings in relation with the endomorphism algebras of the maximal ideals and numerical semigroup rings.
We show that the solution of the Rohrbach problem in additive number theory provides an upper bound for the multiplicity of far-flung Gorenstein numerical semigroup rings. 
Reflexive modules over far-flung Gorenstein rings are also studied.
\end{abstract}

\maketitle



\section{Introduction}\label{section1}

Let $R$ be a Noetherian ring, and let $M$ be a finitely generated $R$-module. Then
\begin{align*} 
\mathrm{tr}_R(M)=\sum_{f\in \Hom_R(M, R)} \Im f = \Im (\mathrm{ev})
\end{align*}
where $\mathrm{ev}:\Hom_R(M, R) \otimes_R M \to R$; $\ f\otimes x\mapsto f(x)$ ($f\in \Hom_R(M, R)$ and $x\in M$) denotes the evaluation map, is called the {\it trace ideal} of $M$. An ideal $I$ of $R$ is called a {\it trace ideal} if $I=\mathrm{tr}_R(M)$ for some $R$-module $M$. 

The notion of trace ideals was recently studied by several papers (\cite{DMP, GIK2, HHS, K, Lin, LP}). In particular, Herzog, Hibi, Stamate deeply studied the trace ideal of the canonical module, and they introduced a new notion, namely, {\it nearly Gorenstein rings} as a class of non-Gorenstein Cohen-Macaulay rings (\cite{HHS}). Here, recall that the trace ideal of the canonical module defines the non-Gorenstein locus of the ring (\cite[after 11.41. Lemma]{LW} and \cite[Lemma 2.1]{HHS}). Thus, we may suppose that the ring is close to being Gorenstein if the trace ideal of the canonical module is large. Indeed, the notion of nearly Gorenstein rings is defined by the inclusion $\tr_R(\omega_R)\supseteq \fkm$ for a Cohen-Macaulay local ring $(R, \fkm)$ possessing the canonical module $\omega_R$ of $R$.

In this paper, we study the opposite case in some sense, that is, the case where the trace ideal of the canonical module is as small as possible. It is known that for a one-dimensional generically Gorenstein Cohen-Macaulay local ring $(R, \fkm)$ possessing the canonical module $\omega_R$, $\tr_R(\omega_R)\supseteq R:\ol{R}$, where $\ol{R}$ denotes the integral closure of $R$. Thus, we call $R$ a {\it far-flung Gorenstein ring} if $\tr_R(\omega_R)=R:\ol{R}$. From the definition, it seems that there is only little that we can expect for far-flung Gorenstein rings. However, we will obtain that far-flung Gorenstein rings enjoy interesting properties. 

To illustrate our results, let $(R, \fkm)$ be a (one-dimensional) far-flung Gorenstein ring. For simplicity, suppose that $R/\fkm$ is infinite and $\ol{R}$ is a local ring. Then, we obtain the bounds 
\[
\rmr(R)+1\le \rme(R) \le \binom{r+1}{2}
\]
for the multiplicity of $R$, where $\rme(R)$ denotes the multiplicity and $r=\rmr(R)$ denotes the Cohen-Macaulay type of $R$ (Corollary \ref{0.11}).  Furthermore, the upper bound can be improved if $R$ is a numerical semigroup ring. In this case, we will see that $\rme(R)\le \ol{n}(r)$, where the integer $\ol{n}(r)$ denotes the solution of the Rohrbach problem for $r$ (Corollary \ref{cor5.3}). Note that the Rohrbach problem is a long-standing problem in additive number theory, see \cite{Ro, Slo}.

We also obtain that if $(R, \fkm)$ is a far-flung Gorenstein ring,  the endomorphism algebra $\Hom_R(\fkm, \fkm)$ of the maximal ideal is again far-flung Gorenstein  (Theorem \ref{b3.2}) and $\Hom_R(\omega_R, M)$ is $\ol{R}$-free for all reflexive modules $M$ of positive rank (Theorem \ref{thm4.1}). Examples arising from numerical semigroup rings are also explored.

The remainder of this paper is organized as follows. In Section \ref{section2} we give a characterization of far-flung Gorenstein rings and prove the bounds $\rmr(R)+1\le \rme(R) \le \binom{r+1}{2}$. In Section 3 we prove Theorem 3.2. 
In Section \ref{section5} we study reflexive modules over far-flung Gorenstein rings. In Section \ref{section5.5} we revisit the bound for the multiplicity of far-flung Gorenstein numerical semigroup rings in relation with the Rohrbach problem. In Section \ref{section6} we study in more details numerical semigroup rings.

Let us fix our notation throughout this paper. In what follows, $(R, \fkm)$ is a Cohen-Macaulay local ring and $M$ is a finitely generated $R$-module. Then 
$\ell_R(M)$, $\mu_R(M)$, and $\rme(M)$ denote the length, multiplicity, and the number of minimal generators of $M$, respectively. $\rmr(R)$ and $v(R)$ denote the Cohen-Macaulay type and embedding dimension of $R$, respectively.

Let $\rmQ(R)$ denote the total ring of fractions of $R$, and let $\ol{R}$ denote the integral closure of $R$. Then a finitely generated $R$-submodule of $\rmQ(R)$ containing a non-zerodivisor of $R$ is called a {\it fractional ideal}. For two fractional ideals $I$ and $J$, $I:J$ denotes the colon ideal of $I$ and $J$ which is given by the set $\{\alpha\in \rmQ(R) \mid \alpha J\subseteq I\}$  (see \cite{HK}).

\begin{acknowledgments}
We would like to thank Mihai Cipu for telling us about the Rohrbach problem.
\end{acknowledgments}


\section{Far-flung Gorenstein rings and bounds of the multiplicity}\label{section2}

Throughout this paper, unless otherwise noted, let $(R, \fkm)$ be a Cohen-Macaulay local ring of dimension one, possessing the canonical module $\omega_R$. 
We set the conditions on $R$ as follows.
\begin{condition} 
\begin{enumerate}[{\rm (a)}] 
\item There exists an $R$-submodule $C$ of $\rmQ(R)$ such that $R\subseteq C\subseteq \ol{R}$ and $C\cong \omega_R$.
\item $\ol{R}$ is finitely generated as an $R$-module.
\item $\ol{R}$ is a local ring with the maximal ideal $\fkn$.
\item There exists a non-zerodivisor $a\in \fkm$ such that $(a)$ is a reduction of $\fkm$.
\end{enumerate}
\end{condition}

\begin{rem} \label{a2.1}
\begin{enumerate}[{\rm (i)}] 
\item If the ring $R$ is generically Gorenstein (i.e., $R_\fkp$ is Gorenstein for all $\fkp \in \Ass R$) and the residue field $R/\fkm$ is infinite, then $R$ satisfies the conditions (a) and (d). Indeed, there exist a canonical ideal $\omega\subsetneq R$ and its reduction $(a)\subseteq \omega$. Hence $C=\frac{\omega}{a}$ is the module satisfying the condition (a).
\item If $R$ is a numerical semigroup ring $K[|H|]=K[|t^h : h\in H|]$, where $K[|t|]$ is formal power series ring over a field $K$ and $H$ is a subsemigroup of $\mathbb{N}_0=\{0, 1, 2, \dots\}$, then $R$ satisfies the conditions (a)-(d) (see Section \ref{section6}). 
\item The condition (d) implies that $\fkm \ol{R}=a\ol{R}$. Indeed, we have $(a)\subseteq \fkm \subseteq \ol{(a)}=a\ol{R}\cap R\subseteq a\ol{R}$, where $\ol{(a)}$ denotes the integral closure of ideal $(a)$ in the sense of \cite[page 2]{SH}. It follows that $R\subseteq \frac{\fkm}{a}\subseteq \ol{R}$; hence, $\frac{\fkm}{a}\ol{R}=\ol{R}$. 
\item For all fractional ideals $I$ and $J$, 
\[
I:J\cong \Hom_R(J, I)
\] 
(\cite[Lemma 2.1]{HK}). In particular, with the condition (a), we have $C:(C:I)=I$ (see \cite[Definition 2.4]{HK}).
\end{enumerate} 
\end{rem}

Under condition (a), we can regard the evaluation map
\begin{align*} 
\mathrm{ev}: \Hom_R(\omega_R, R)\otimes_R \omega_R \to R 
\end{align*}
as the multiplication map $(R:C) \otimes_R C \to R$; $f\otimes x \to fx$ ($f\in R:C$ and $x\in C$). It follows that $\tr_R(\omega_R)=(R:C)C$.
The following is a starting point of this paper.

\begin{Lemma}\label{a2.2} {\rm (cf. \cite[Proposition A.1]{HHS2})}
With the conditions {\rm (a)} and {\rm (b)}, we have $R:\ol{R}\subseteq \tr_R(\omega_R)$.
\end{Lemma}

\begin{proof}
This follows from the fact that $\tr_R(\omega_R)=(R:C)C\supseteq R:C\supseteq R:\ol{R}$.
\end{proof}

Since the trace ideal of the canonical module defines the non-Gorenstein locus of the ring (\cite[after 11.41. Lemma]{LW}, \cite[Lemma 2.1]{HHS}), we may suppose that the ring is close to being Gorenstein if the trace ideal of the canonical module is large. In this paper, we study the opposite case, that is, the case where the trace ideal of the canonical module is as small as possible:

\begin{Definition}\label{defffg}
Suppose that $R$ satisfies the conditions {\rm (a)} and {\rm (b)}.
We say that $R$ is a {\it far-flung Gorenstein ring} if $\tr_R(\omega_R)=R:\ol{R}$.
\end{Definition}

\begin{rem} 
Recall that for an arbitrary Cohen-Macaulay local ring $(R, \fkm)$ possessing the canonical module $\omega_R$, we say that $R$ is a {\it nearly Gorenstein ring} if $\tr_R(\omega_R)\supseteq \fkm$ (\cite{HHS}). Hence, if a one-dimensional Cohen-Macaulay local ring $R$ satisfies the conditions {\rm (a)} and {\rm (b)}, $R$ is nearly Gorenstein and far-flung Gorenstein if and only if $R:\ol{R}\supseteq \fkm$. This is equivalent to saying that $\ol{R}= \fkm:\fkm (\cong \Hom_R(\fkm, \fkm))$. Indeed, we may assume that $R$ is not a discrete valuation ring. We then obtain that 
\[
R:\ol{R}\supseteq \fkm \Leftrightarrow \fkm \ol{R}\subseteq R \Leftrightarrow \ol{R}\subseteq R:\fkm=\fkm :\fkm,
\]
where the equality $R:\fkm=\fkm :\fkm$ is well understood and described in Lemma \ref{b3.2}.

In particular, all nearly Gorenstein far-flung Gorenstein rings are almost Gorenstein rings in the sense of \cite[Definition 3.1]{GMP} and have minimal multiplicity (see \cite[Theorem 5.1]{GMP}).
Furthermore, if $R=K[|H|]$ is a numerical semigroup ring, then the condition $R:\ol{R}\supseteq \fkm$ is equivalent to saying that $H=\left<n, n+1, \dots, 2n-1\right>$ for some $n>0$. 
\end{rem}

The following is a characterization of far-flung Gorenstein rings.

\begin{thm} \label{0.8}
Suppose that $R$ satisfies the conditions {\rm (a)-(c)}.
Then the following are equivalent:
\begin{enumerate}[{\rm (i)}] 
\item $R$ is a far-flung Gorenstein ring;
\item $\tr_R(\omega_R)\cong \ol{R}$;
\item $C^2=\ol{R}$.
\end{enumerate} 
\end{thm}

\begin{proof}
(i) $\Rightarrow$ (ii): Note that $R:\ol{R}$ is a nonzero ideal of $\ol{R}$. Since $\ol{R}$ is a discrete valuation ring (for example, see \cite[Theorem 2.2.22]{BH}), it follows that $\tr_R(\omega_R)=R:~\ol{R}\cong \ol{R}$.

(ii) $\Rightarrow$ (i): $\tr_R(\omega_R)=\alpha \ol{R}$ for some $\alpha\in \rmQ(R)$. We then have 
\[
\alpha\in \alpha \ol{R}=\tr_R(\omega_R)\subseteq R,
\]
thus $\alpha \ol{R}$ is an ideal of $\ol{R}$ in $R$. Hence, $\tr_R(\omega_R)=\alpha \ol{R}\subseteq R:\ol{R}$. The reverse inclusion follows from Lemma \ref{a2.2}.

(i) $\Rightarrow$ (iii): This follows from the following implications.
\begin{align*}
\text{$R$ is a far-flung Gorenstein ring} &\Leftrightarrow R:\ol{R} =\tr_R(\omega_R) &&\Leftrightarrow R:\ol{R} =(R:C)C \\
&\Rightarrow R:\ol{R}=R:C &&\Leftrightarrow (C:C):\ol{R}=(C:C):C \\
&\Leftrightarrow C:C\ol{R}=C:C^2 &&\Leftrightarrow C^2=C\ol{R}=\ol{R},
\end{align*}
where the third implication follows from the inclusions $R:\ol{R}  \subseteq R:C \subseteq (R:C)C$ and the sixth equivalence follows by applying the $C$-dual $\Hom_R(-, C)=C:-$.

(iii) $\Rightarrow$ (i): By the argument of (i) $\Rightarrow$ (iii), $C^2=\ol{R}$ implies that $R:\ol{R}=R:C$. Hence, 
\[
\tr_R(\omega_R)=(R:C)C=(R:\ol{R})C=R:\ol{R},
\]
where the third equality follows from $R:\ol{R}\subseteq (R:\ol{R})C\subseteq (R:\ol{R})\ol{R}=R:\ol{R}$.
\end{proof}

By using Theorem \ref{0.8} we find an upper bound of the multiplicity for far-flung Gorenstein rings. 
Recall that there is a lower bound of the multiplicity of Cohen-Macaulay local rings.

\begin{Fact} {\rm (cf. \cite[3.1 Proposition]{S3})}\label{b2.5}
If $(R, \fkm)$ is a Cohen-Macaulay local ring with $\rme(R)>1$ (not necessarily of dimension one), then $\rmr(R)+1 \le \rme(R)$. Furthermore, $\rmr(R)+1 = \rme(R)$ holds if and only if $R$ has minimal multiplicity.
\end{Fact}

\begin{proof}
We may assume that $R/\fkm$ is infinite by passing to $R \to R[X]_{\fkm R[X]}$. Then we can choose a parameter ideal $Q$ as a reduction of $\fkm$. Hence we obtain that
\[
\rme(R)-1=\ell_R (R/Q)-\ell_R (R/\fkm)=\ell_R (\fkm/Q) \ge \ell_R ((Q:_R\fkm)/Q)=\rmr (R).
\] 
The equality holds true if and only if $\fkm = Q:_R\fkm$, which is equivalent to saying that $\fkm^2 =Q\fkm $ by \cite[Theorem 2.2.]{CP}.
\end{proof}

\begin{Corollary}\label{0.11}
Suppose that $R$ satisfies the conditions {\rm (a)-(d)}. If $R$ is a far-flung Gorenstein ring, then the inequalities
\[
\rmr(R)+1 \le \rme(R) \le \binom{\rmr(R)+1}{2}
\] 
hold.
\end{Corollary}

\begin{proof}
By Theorem \ref{0.8}, $\mu_R(\ol{R})=\mu_R(C^2)\le \binom{\rmr(R)+1}{2}$. Meanwhile, we have $\mu_R(\ol{R})=\ell_R(\ol{R}/\fkm \ol{R})=\ell_R(\ol{R}/a\ol{R})$ by Remark \ref{a2.1}(iii). Since $(a)$ is a reduction of $\fkm$, $\ell_R(\ol{R}/a\ol{R})=\rme(\ol{R})$. Therefore, since $\rme(\ol{R})=\rme(R)+ \rme(\ol{R}/R)=\rme(R)$ by the additivity of the multiplicity, we obtain that $\mu_R(\ol{R})=\rme(R)$.
\end{proof}

\begin{Corollary}\label{cor:type2}
Suppose that $R$ satisfies the conditions {\rm (a)-(d)}. If $R$ is a far-flung Gorenstein ring of type $2$, then $R$ has minimal multiplicity. Moreover, under these hypotheses, the multiplicity of $R$ is $3$. 
\end{Corollary}

\begin{proof}
Let $R$ be a far-flung Gorenstein ring with  $\rmr(R)=2$.  Then we have $3=\rmr(R)+1 \le \rme(R) \le \binom{\rmr(R)+1}{2}=3$ by Corollary \ref{0.11}. Hence $R$ has minimal multiplicity of $3$ by Fact \ref{b2.5}.  
\end{proof}

\begin{ex}\label{b2.8}
Let $R=K[|H|]$ be a numerical semigroup ring with $\rmr(R)=2$. Then $R$ is a far-flung Gorenstein ring if and only if $H=\langle 3, 3n+1, 3n+2\rangle$ for some integer $n>0$.
\end{ex}

\begin{proof}
If $R$ is a far-flung Gorenstein ring of type $2$, Corollary~\ref{cor:type2} implies $v(R)=\rme(R)=3$. 
It is known from  \cite[Proposition 2.5]{HHS2} that when $H$ is $3$-generated and not symmetric, the ring $K[|H|]$ is far-flung Gorenstein if and only if $H=\langle 3, 3n+1, 3n+2\rangle$ for some integer $n >0$.
\end{proof}

In contrast to Corollary~\ref{cor:type2}, far-flung Gorenstein rings need not have minimal multiplicity in general.

\begin{ex}\label{b2.9}
Let $R=K[|t^7, t^8, t^{11}, t^{17}, t^{20}|]$. Then $\tr_R(\omega_R)=R:\ol{R}=t^{14}\ol{R}$. Thus $R$ is far-flung Gorenstein, but $R$ does not have minimal multiplicity: $\rme(R)=7>v(R)=5$.
\end{ex}

Although Corollary \ref{0.11} gives an upper bound for the multiplicity, this is not sharp for numerical semigroup rings. We will return to this topic in Section \ref{section5.5}.


\section{The endomorphism algebra $\Hom_R(\fkm, \fkm)$}\label{section4}

Throughout this section, let $(R, \fkm)$ be a Cohen-Macaulay local ring of dimension one possessing the canonical module $\omega_R$, and suppose that $R$ satisfies the conditions {\rm (a)-(d)}. Here, we study the far-flung Gorenstein property of the endomorphism algebra $\Hom_R(\fkm, \fkm)$ in relation with that of $R$. Set $B=\fkm:\fkm\cong \Hom_R(\fkm, \fkm)$. The following lemmas are known, but we include a proof for the convenience of readers.

\begin{Lemma}\label{b3.1}
$B$ is a Cohen-Macaulay local ring of dimension one and $\ol{B}=\ol{R}$.
\end{Lemma}

\begin{proof}
$B$ is a subring of $\ol{R}$ and finitely generated as an $R$-module. Hence $\ol{B}=\ol{R}$ holds. 
Furthermore, any maximal ideal of $B$ is forced to be $\fkn\cap B$. It follows that $B$ is a local ring with the maximal ideal $\fkn\cap B$. 
\end{proof}

\begin{Lemma}\label{b3.2}
If $R$ is not a discrete valuation ring, then $R\subsetneq B=R:\fkm$.
\end{Lemma}

\begin{proof}
Suppose that there exists an element $f\in (R:\fkm)\setminus (\fkm:\fkm)$. Then $\fkm f\subseteq R$ but $\fkm f\not\subseteq \fkm$. It follows that $\fkm f =R$ and hence $\fkm$ is cyclic. This shows that $R$ is a discrete valuation ring, which is a contradiction. Thus $R:\fkm=B$. To prove $R\subsetneq B$, consider the short exact sequence $0\to \fkm \to R \to R/\fkm \to 0$. By applying the functor $\Hom_R(-, R)\cong R:-$, we obtain that 
\[
0 \to R=R:R \to B=R:\fkm \to \Ext_R^1(R/\fkm, R) \to 0. 
\]
Therefore, we have $\ell_R(B/R)=\rmr(R)>0$.
\end{proof}

\begin{Theorem}\label{b3.3}
Suppose that $R$ satisfies {\rm (a)-(d)}. If $R$ is a far-flung Gorenstein ring, then $B\cong \Hom_R(\fkm, \fkm)$ is also a far-flung Gorenstein ring.
\end{Theorem}

\begin{proof}
If $R$ is a discrete valuation ring, then the assertion of the theorem is obvious because $R=B$.  Hence we may assume that $R$ is not a discrete valuation ring. Then the assertion follows from the following claim.
\end{proof}

\begin{claim}\label{claim1}
Suppose that $R$ satisfies {\rm (a)-(d)} and $R$ is a far-flung Gorenstein ring. Choose  an element $a \in \fkm$ such that $(a)$ is a reduction of $\fkm$ {\rm (}recall the condition {\rm (d)}{\rm )}. 
Then the following assertions hold true.
\begin{enumerate}[{\rm (i)}] 
\item $B:\ol{B}=\frac{1}{a} (R:\ol{R})$.
\item $\tr_B(\omega_B)=\frac{1}{a} \tr_R(\omega_R)$.
\end{enumerate} 	
\end{claim}

\begin{proof}[Proof of Claim \ref{claim1}]
(i): This follows from 
\begin{align*}
B:\ol{B}=B:\ol{R}=(R:\fkm):\ol{R}=R:\fkm\ol{R}=R:a\ol{R}=\frac{1}{a} (R:\ol{R})
\end{align*}
by Remark \ref{a2.1}(iii) and Lemmas \ref{b3.1} and \ref{b3.2}.

(ii): Since we have $\omega_B\cong \Hom_R(B, C)\cong C:B$, we obtain that 
\[
\tr_B(\omega_B)=(B:(C:B))(C:B).
\] 
By noting that $B=R:\fkm=(C:C):\fkm=C:\fkm C$, we obtain that $C:B=C:~(C:\fkm C)=\fkm C$. Hence we have $\tr_B(\omega_B)=(B:\fkm C)\fkm C$.

Furthermore, we can compute $B:\fkm C$ as follows:
\begin{align*}
B:\fkm C&=(R:\fkm):\fkm C=R:\fkm^2 C = (R:C) : \fkm^2\\
&=(R: \ol{R}):\fkm^2= R:\fkm^2 \ol{R}=R:a^2\ol{R}=\frac{1}{a^2} (R:\ol{R})\\
&=\frac{1}{a^2} (R:C),
\end{align*}
where the fourth and the eighth equalities follow from the fact $R:C=R:\ol{R}$ (see the proof of Theorem \ref{0.8} (i)\implies (iii)). 
It follows that 
\begin{align*}
\tr_B(\omega_B)=\frac{1}{a^2}(R:C)\fkm C=\frac{\fkm}{a^2}\tr_R(\omega_R)=\frac{\fkm}{a^2}(R:\ol{R})=\frac{a}{a^2}(R:\ol{R})=\frac{1}{a}\tr_R(\omega_R).
\end{align*}
\end{proof}

The converse of Theorem \ref{b3.3} is not true.

\begin{ex} 
Let $R=K[|t^4, t^5, t^6|]$ be a numerical semigroup ring, where $K$ is a field. Let $\fkm$ be the maximal ideal of $R$. Then $R$ is not far-flung Gorenstein by Example \ref{b2.8}. But 
\[
\fkm:\fkm=K[|t^4, t^5, t^6, t^7|]
\]
is a far-flung Gorenstein ring by Corollary \ref{ffg-minmult-aseq}.
\end{ex}


\section{Reflexive modules over far-flung Gorenstein rings}\label{section5}

Let $(R, \fkm)$ be a Cohen-Macaulay local ring of dimension one possessing the canonical module $\omega_R$. 
Suppose that $R$ satisfies the conditions {\rm (a)-(c)}. If $R$ is a far-flung Gorenstein ring, then we observe that by applying the functor $\Hom_R(\omega_R, -)$ to $R$
\[
\Hom_R(\omega_R, R) \cong R:C = R:\ol{R} \cong \ol{R}.
\]
This observation can be generalized in the following way. Let $(-)^*$ and $(-)^\vee$ denote the $R$-dual $\Hom_R(-, R)$ and the canonical dual $\Hom_R(-, \omega_R)$, respectively.

\begin{thm} \label{thm4.1}
Let $R$ be a far-flung Gorenstein ring, and let $M$ be a reflexive module of rank $r>0$.  Then $\Hom_R(\omega_R, M) \cong \ol{R}^r$.
\end{thm}

\begin{proof} 
Since $M$ is reflexive, there exists an exact sequence 
$0 \to M \to F_1 \to F_2,$
where $F_1$ and $F_2$ are finitely generated free $R$-modules (see, for example, \cite[Proposition~4.1]{HKS}). Hence, by applying the canonical dual $(-)^\vee$, we obtain a surjection 
\[
F_1^\vee \cong \omega_R^{\rank_R F_1} \to M^\vee
\]
because $X$ is a maximal Cohen-Macaulay $R$-module, where $X$ is the image of the map $F_1 \to F_2$.  It follows that $\tr_R(\omega_R)\supseteq \tr_R(M^\vee)$ by \cite[Proposition 2.8(i)]{Lin}. Note that $\tr_R(M^\vee)$ contains a non-zerodivisor. Indeed, $\tr_R(M^\vee)_\fkp=\tr_{R_\fkp} ((M^\vee)_\fkp)=R_\fkp$ for all associated prime ideals  $\fkp$ by \cite[Proposition 2.8(viii)]{Lin} since $M^\vee$ has a positive rank. It follows that $\tr_R(M^\vee)\not\subseteq \fkp$ for all associated prime ideals  $\fkp$.

Therefore, we obtain that 
\[
\ol{R} \subseteq R:(R:\ol{R})=R:\tr_R(\omega_R) \subseteq R:\tr_R(M^\vee) =\tr_R(M^\vee) : \tr_R(M^\vee) \subseteq \ol{R},
\]
where the fourth equality follows from \cite[Corollary 2.2]{GIK2}.
On the other hand, we have a surjection
\[
(M^\vee)^* \otimes_R M^\vee \xrightarrow{\mathrm{ev}} \tr_R(M^\vee) \to 0, 
\]
where $\mathrm{ev}: f \otimes x \mapsto f(x)$ ($f\in (M^\vee)^*$ and $x\in M^\vee$), by definition of the trace ideal of $M^\vee$. Therefore, by applying the $R$-dual to the above surjection, we obtain that 
\[
0 \to R: \tr_R(M^\vee) = \ol{R} \to \Hom_R((M^\vee)^*, (M^\vee)^*).
\]
Since $(M^\vee)^*$ is a reflexive module, we can regard $(M^\vee)^*$ as an $\ol{R}$-module (see \cite[Proposition 2.4]{IK} or \cite[(7.2) Proposition]{Ba}). Since $\ol{R}$ is a discrete valuation ring, it follows that $(M^\vee)^* \cong (\ol{R})^r$.
This completes the proof since we have the isomorphisms
\[
(M^\vee)^* \cong \Hom_R(M^\vee \otimes_R \omega_R, \omega_R) \cong \Hom_R(\omega_R, M^{\vee \vee})\cong \Hom_R(\omega_R, M).
\] 
\end{proof}

\section{Revisiting the upper bound of the multiplicity}\label{section5.5}

Although Corollary \ref{0.11} gives an upper bound for the multiplicity, it is not sharp. In what follows, we investigate a sharp upper bound of the multiplicity of far-flung Gorenstein numerical semigroup rings.
Let $(R, \fkm)$ be a Cohen-Macaulay local ring of dimension one possessing the canonical module $\omega_R$. 
Suppose that $R$ satisfies the conditions (a)-(d).
In this section, we further assume that the canonical ring homomorphism $R/\fkm \to \ol{R}/\fkn$ is bijective. 
For an element $x\in \ol{R}$, let 
\[
v(x) = \ell_{\ol{R}}(\ol{R}/x\ol{R})
\]
denote the {\it discrete valuation} of $x$.
Choose $f_1, f_2, \dots, f_r\in C$ such that 
\[
C=\langle f_1, f_2, \dots, f_r\rangle,
\] 
where $r=\rmr(R)$. 
Set $n_i=v(f_i)$ for $1\le i \le r$. We may assume that $n_1 \le n_2 \le \cdots \le n_r$ after replacing in the appropriate order. 

Assume that $n_i=n_j$ for some $1\le i < j \le r$. Then we have $f_i \ol{R}=f_j \ol{R}=\fkn^{n_i}$. Hence $f_j=f_i x$ for some $x\in \ol{R}\setminus \fkn$. Since we have the isomorphism $\varphi: R/\fkm \to \ol{R}/\fkn$, there exists $y\in R\setminus \fkm$ such that $x-y\in \fkn$. Hence 
\[
f_j=(x-y+y)f_i=(x-y)f_i+yf_i, \quad \text{i.e., $f_j-yf_i=(x-y)f_i$}.
\]
Therefore, by replacing $f_j$ with $f_j-yf_i$ among the generators of $C$, we may assume that $n_i<n_j$. It follows that we can choose a minimal system of generators $f_1, f_2, \dots, f_r$ of $C$ such that 
\begin{align} \label{eq1}
n_1 < n_2 < \cdots < n_r.
\end{align}

With this notation introduced we have the following.

\begin{thm} \label{b5.1}
Consider the following assertions:
\begin{enumerate}[{\rm (i)}] 
\item $R$ is a far-flung Gorenstein ring.
\item There exists a fractional canonical module $C=\left<f_1, f_2, \dots, f_r\right>$, where $n_i=v(f_i)$ satisfies the inequalities {\rm (\ref{eq1})}, such that the sum-set $\{n_i+n_j : 1\le i\le j \le r\}$ contains the integers $0, 1, \dots, \rme(R)-1$.
\end{enumerate} 
Then {\rm (ii) $\Rightarrow$ (i)} holds. The implication {\rm (i) $\Rightarrow$ (ii)} also holds if $R$ is a numerical semigroup ring $K[|H|]$. When the condition {\rm (ii)} holds, then $n_1=0$ and $n_2=1$.
\end{thm}

\begin{proof} 
(ii) $\Rightarrow$ (i): Let $m$ be a non-negative integer, and write 
\[
m=s{\cdot}\rme(R) + n
\] 
for some integers $s\ge 0$ and $0\le n < \rme(R)$. Let $f\in C^2$ be an element with $v(f)=n$. 
By using Remark \ref{a2.1}(iii) and the proof of Corollary \ref{0.11}, we have $v(a)=\mu_R(\ol{R})=\rme(R)$, where $(a)$ is a reduction of $\fkm$. Hence $a^s f\in C^2$ is an element with $v(a^s f)=s{\cdot}\rme(R) + n=m$. It follows that $C^2=\ol{R}$ by \cite{HK2} (see also \cite[Proposition 1]{Matsu}).

(i) $\Rightarrow$ (ii): Suppose that $R=K[|H|]$ is a numerical semigroup ring. Then a fractional canonical module of $R$ is given by 
\[
C=\left<t^{{\rm F}(H)-\alpha} : \alpha \in {\rm PF}(H)\right>
\]
where ${\rm PF}(H)=\{x\in  \ZZ \setminus H: x+h\in H \text{ for all } 0\neq h \in H\}$ is the set of pseudo-Frobenius numbers of $H$ and ${\rm F}(H)=\max\{n\in \mathbb{Z} : n\not\in H\}$ denotes the Frobenius number of $H$ (see \cite[Example (2.1.9)]{GW} and Section \ref{section6}). Note that $C$ is a module satisfying the condition (a). Therefore, if $R$ is a far-flung Gorenstein ring, then $C^2=\ol{R}$ by Theorem \ref{0.8}. 

On the other hand, $v(x)\ge \rme(R)$ for all $x\in \fkm$ since 
\begin{align*} 
\ell_{\ol{R}} (\ol{R}/x \ol{R}) \ge \ell_{\ol{R}} (\ol{R}/\fkm \ol{R})=\ell_R (\ol{R}/\fkm \ol{R})= \mu_R(\ol{R})=\rme(R),
\end{align*}
where the second equality follows from the assumption that $R/\fkm \cong \ol{R}/\fkn$, the third equality follows by Remark \ref{a2.1}(iii), the fourth equality follows from the proof of Corollary \ref{0.11}. 

Hence, by noting that $C^2$ is generated by monomials, for each $0\le n <\rme(R)$, an element $f$ with $v(f)=n$ can be chosen as $t^{\rmF(H)-\alpha}{\cdot}t^{\rmF(H)-\beta}$, where $\alpha, \beta\in \mathrm{PF}(H)$. It follows that by setting $f_i=t^{\rmF(H)-\alpha_{r-i}}$, where ${\rm PF}(H)=\{\alpha_1<\alpha_2<\cdots <\alpha_r=\mathrm{F}(H)\}$, we obtain the assertion (ii).

In particular, there exist $1\le i \le j \le r$ such that $v(f_i{\cdot}f_j)=1$, i.e., $n_i+n_j=1$. It follows that $n_1=0$ and $n_2=1$.
\end{proof}

The assertion (ii) of Theorem \ref{b5.1} is tightly related to the Rohrbach problem:

\begin{prob} (The Rohrbach problem, \cite{Ro, Slo})
Let $A$ be a set of non-negative integers with $r$ elements. Let $n(A)$ denote the integer such that 
the sum-set 
\[
A+A=\{a+b:a,b\in A\}
\] contains the integers $0, 1, \dots, n(A)-1$ but not $n(A)$. If $0\notin A$ then let $n(A)=-1$.
For $r>0$ the Rohrbach problem asks to find the integer 
\[
\overline{n}(r)=\max\{ n(A):  |A|=r\}. 
\] 
\end{prob}

With this notation we have the following result, which improves Corollary \ref{0.11}.

\begin{cor} \label{cor5.3}
Suppose that $R$ is a numerical semigroup ring. If $R$ is a far-flung Gorenstein ring, then the inequality
\[
\rme(R) \le \ol{n}(r)
\] 
holds, where $r=\rmr(R)$.
\end{cor}

The solution $\ol{n}(r)$ of the Rohrbach problem is known for $r\le 25$ (see \cite{KoCo, Slo}):

\begin{table}[htb]
  \begin{tabular}{|c||rrrrrrrrrrrrr|} \hline
$r$            & 1 & 2 & 3 & 4 & 5 & 6 & 7 & 8 & 9 & 10 & 11 & 12 & 13\\ 
$\ol{n}(r)$ &1 & 3 & 5 & 9 & 13 & 17 & 21 & 27 & 33 & 41 & 47 & 55 & 65 \\ \hline
$r$            & 14 & 15 & 16 & 17 & 18 & 19 & 20 & 21 & 22 & 23 & 24 & 25 & $\cdots$\\ 
$\ol{n}(r)$ &73 & 81 & 93 & 105 & 117 & 129 & 141 & 153 & 165 & 181 & 197 & 213 & $\cdots$\\ \hline
  \end{tabular}
\end{table}

 The following example ensures that $\ol{n}(r)$ provides the sharp upper bounds of the multiplicity of far-flung Gorenstein numerical semigroup rings, if $r=\rmr(R)\le 5$.

\begin{ex} 
Let $K$ be a field and $K[|t|]$ denote the formal power series ring. Then by using Proposition \ref{b6.1} it is easy to check that the following hold true.
\begin{enumerate}[{\rm (i)}] 
\item Let $R_1=K[|t^5, t^6, t^{13}, t^{14}|]$. Then $R_1$ is a far-flung Gorenstein ring with $\rmr(R_1)=3$ and $\rme(R_1)=5$.
\item Let $R_2=K[|t^9, t^{10}, t^{11}, t^{12}, t^{15}|]$. Then $R_2$ is a far-flung Gorenstein ring with $\rmr(R_2)=4$ and $\rme(R_1)=9$.
\item Let $R_3=K[|t^{13}, t^{14}, t^{15}, t^{16}, t^{17}, t^{18}, t^{21}, t^{23}|]$. Then $R_3$ is a far-flung Gorenstein ring with $\rmr(R_3)=5$ and $\rme(R_1)=13$.
\end{enumerate} 
\end{ex}

\begin{quest}
\begin{enumerate}[{\rm (i)}] 
\item Is it true that for any $r\geq 2$ there exists a far flung Gorenstein numerical semigroup ring $R$ of type $r$ with $\rme(R)=\overline{n}(r)$?
\item For any $r\geq 2$, what is the exact range of values of $\rme(R)$ when $R$ runs over all far-flung Gorenstein numerical semigroup rings of type $r$? 
\end{enumerate}
\end{quest}

\section{far-flung Gorenstein numerical semigroup rings}\label{section6}

In this section we investigate numerical semigroup rings. When a ring is a numerical semigroup ring, we can check the far-flung Gorenstein property by using only the data of the numerical semigroup (Proposition \ref{b6.1}). First, let us recall some basic notation of numerical semigroup rings. 

A numerical semigroup $H$ is a submonoid of $\NN_0$ such that $\NN_0\setminus H$ is finite.
The set 
\[
{\rm PF}(H)=\{x\in  \ZZ \setminus H: x+h\in H \text{ for all } 0\neq h \in H\}
\] 
is called the set of the {\it pseudo-Frobenius numbers} of $H$. 
The largest value in ${\rm PF}( H)$ is called the {\it Frobenius number} of $H$, denoted ${\rm F}(H)$. The smallest nonzero element in $H$ is called the {\it multiplicity} of $H$ and we denote it by ${\rm e}(H)$.

Let $K$ be a field. Then, the pseudo-Frobenius numbers of $H$ define the canonical module of the numerical semigroup ring $R=K[|H|]$. Indeed, it is known that a graded canonical module of $S=K[H]$ is $\omega_S=\sum_{\alpha \in {\rm PF}(H)} S t^{-\alpha}$, see \cite[Example (2.1.9)]{GW}. After multiplication by $t^{{\rm F}(H)}$ we obtain the canonical module
\[
D=\sum_{\alpha \in {\rm PF}(H)}  St^{{\rm F}(H)-\alpha}
\] 
such that $S\subseteq D\subseteq \ol{S}=K[t]$. By noting that $R=K[|H|]$ is the completion of $S_\fkM$, where $\fkM$ is the graded maximal ideal of $S$, we obtain that
\[
C=\sum_{\alpha \in {\rm PF}(H)}  Rt^{{\rm F}(H)-\alpha}
\]
is a module appearing in the condition (a). It is known that $\rme(R)=\rme(H)$. With this notation we have the following.

\begin{Proposition}\label{b6.1}
\label{prop:ffg-h}
Let $R=K[|H|]$ be the semigroup ring for the numerical semigroup $H$. The following statements are equivalent:
\begin{enumerate}[{\rm (i)}] 
\item  $R$ is a far-flung Gorenstein ring.
\item $\{0, \dots, {\rm e}(H)-1\} \subseteq  \{ 2 {\rm{F}(H)}-\alpha-\beta: \alpha, \beta \in {\rm PF}(H)\}$.
\item $\{2{\rm F}(H)- {\rm e}(H)+1,\dots, 2{\rm F}(H)\} \subseteq  \{\alpha+\beta: \alpha, \beta\in {\rm PF}(H)\}$.
\end{enumerate}
\end{Proposition}

\begin{proof}
According to Theorem \ref{0.8}, $R$ is far-flung Gorenstein if and only if $C^2=K[|t|]$.  
This proves that (i) $\Leftrightarrow$ (ii) since $C^2$ is generated by monomials. 
Simple algebraic manipulations show that (ii) $\Leftrightarrow$ (iii). 
\end{proof}

\begin{Corollary}\label{ffg-minimal}
Let $H$ be a numerical semigroup minimally generated by 
$a_1<\dots< a_v$ which is of minimal multiplicity, i.e. $v=a_1$.  Then $K[|H|]$ is a far-flung Gorenstein ring if and only if 
\begin{equation} \label{inclusion}
\{2 a_v-a_1+1,\dots, 2 a_v\}\subseteq \{a_i+a_j:2\leq i, j\leq v\}.
\end{equation}
\end{Corollary}

\begin{proof}
Since $H$ has minimal multiplicity we have 
${\rm PF}(H)=$ $\{ a_2-a_1, \dots, a_v-a_1\}$ and $a_1=v$
 (see \cite{RS}).  The conclusion now follows from Proposition \ref{prop:ffg-h}(iii).
\end{proof}

This corollary has an immediate application, where we find far-flung numerical semigroup rings with minimal multiplicity   of arbitrary embedding dimension.

\begin{Corollary}\label{ffg-minmult-aseq}
Let $a\geq 3$ and $d$ be coprime nonnegative integers and $H=\langle a, a+d, \dots, a+(a-1)d \rangle$. Then $R=K[|H|]$ is a far-flung Gorenstein ring if and only if $d=1$.
\end{Corollary} 
\begin{proof}
Assume that $R$ is far-flung Gorenstein. Then the sum of any two generators of $H$ larger then $a$ is of the form $2a+ jd$, hence we can not generate with them all the integers in any interval of length $a$ unless $d=1$. Conversely, when $d=1$  we verify that \eqref{inclusion} holds: on the right hand side  we have the set $\{2a+2, \dots, 4a-2\}$, which clearly contains the set $\{3a-1,\dots, 4a-2\}$. Thus, $R$ is a far-flung Gorenstein ring. 
\end{proof}

We shall now describe the far-flung Gorenstein numerical semigroup rings of type $3$, not of minimal multiplicity. Recall that the classification of the far-flung Gorenstein numerical semigroup rings of type $2$ is given in Example \ref{b2.8}. 

\begin{Theorem}
Let $R=K[|H|]$ be a numerical semigroup ring. Suppose that $R$ is of type 3 and not of minimal multiplicity. Then the following conditions are equivalent:
\begin{enumerate}[{\rm (i)}] 
\item $R$ is a far-flung Gorenstein ring.
\item \begin{enumerate}
\item[{\rm (1-1)}] $H=\langle 5, 5m+4, 10m+6, 10m+7 \rangle$,  where $m\geq 1$;
\item[{\rm (1-2)}] $H=\langle 5, 5m+1, 10m+3, 10m+4 \rangle$,  where $m\geq 1$; 
\item[{\rm (2-1)}] $H=\langle 5, 5m+2, 10m+1, 10m+3 \rangle$,  where $m\geq 1$;
\item[{\rm (2-2)}] $H=\langle 5, 5m+3, 10m+4, 10m+7\rangle$,  where $m\geq 1$.
\end{enumerate}
\end{enumerate}
\end{Theorem}

\begin{proof}
(i) $\Rightarrow$ (ii): By Corollaries \ref{0.11} and \ref{cor5.3}, we have $\rme(R)=5$. Hence $\embdim R=4$ because  $\embdim R<\rme(R)$ and $\embdim R>3$ by \cite{H}. Furthermore, $C =\sum_{\alpha\in \mathrm{PF}(H)} Rt^{\mathrm{F}(H)-\alpha}$ is generated by either $1, t, t^2$ or $1, t, t^3$.

Assume that $C= \langle 1, t, t^2 \rangle$. Then, since $\rme(R)=5$, we have either the case  
\begin{align*}
\text{{\rm (1-1):}}& \quad \mathrm{PF}(H)=\{ 5n+1, 5n+2, 5n+3\} \quad \text{or}\\
\text{{\rm (1-2):}}& \quad \mathrm{PF}(H)=\{ 5n+2, 5n+3, 5n+4\}
\end{align*}
for some $n\ge 0$. 

{\rm (1-1):} Considering the smallest nonzero element in $H$ in each congruence class modulo $5$,  there exists $m>0$ so that $H=\langle 5, 5n+6, 5n+7, 5n+8, 5m+4 \rangle$  and $5m+4-5=5m-1\not\in H$. 
The pseudo-Frobenius numbers of $H$ are also the elements in $\ZZ \setminus H$ which are maximal with respect to the partial order $\leq_H$ induced by $H$, where $a\leq_H b$ if and only if $b-a \in H$. 
Since $5m-1\not\in \mathrm{PF}(H)$, we get  $5m-1 \leq_H 5n+1$, or $5m-1\leq_H 5n+2$, or $5m-1\leq_H 5n+3$, i.e. 
\begin{align*}
\begin{cases}
5n+3-(5m-1)\in H \quad \text{or} \\
5n+2-(5m-1)\in H \quad \text{or} \\
5n+1-(5m-1)\in H. 
\end{cases}
\end{align*}
Considering the smallest elements in $H$  modulo 5, it is equivalent to saying that
\begin{align*}
\begin{cases}
5n+3-(5m-1)\ge 5m+4 \quad \text{or} \\
5n+2-(5m-1)\ge 5n+8 \quad \text{or} \\
5n+1-(5m-1)\ge 5n+7.
\end{cases}
\end{align*}
  Of these, only $5n+3-(5m-1)\ge 5m+4$ can happen. Hence, we get $2m\le n$.  Thus $5n+8= 2(5m+4)+5(n-2m)$, which gives 
\[
H= \langle 5,  5m+4 , 5n+6, 5n+7\rangle.
\]

The fact that $5n+3 \notin H$ is equivalent to $5n+3\notin \langle 5, 5m+4\rangle$. On the other hand, for any positive integer $p$ one has $5p+3\notin \langle 5, 5m+4\rangle$ if and only if there exist $u,v$ nonnegative integers such that $5p+3=5u+(5m+4)v$. Arguing modulo $5$, the previous equation has nonnegative integer solutions if and only if $5p+3\geq 2(5m+4)$, equivalently $p\geq 2m+1$.
Therefore, $n\leq 2m$, which gives that $n=2m$ and 
$$
H=\langle 5, 5m+4, 10m+6, 10m+7 \rangle, \text{ for }m\geq 1.
$$
Conversely, it is routine to check that such a semigroup $H$ is minimally generated by these four numbers and that ${\rm PF}(H)\supseteq\{10m+1, 10m+2, 10m+3\}$. Since $H$ is not of minimal multiplicity, the previous inclusion is an equality. This completes the analysis of the case (1-1).

The rest of the proof proceeds in a similar  way.

{\rm (1-2):} Arguing modulo $5$ we find $m>0$ so that $H=\langle 5, 5m+1, 5n+7, 5n+8, 5n+9 \rangle$ and $5m+1-5\not\in H$. Since $5m-4\not\in \mathrm{PF}(H)$, we get that
\begin{align*}
\begin{cases}
5n+2-(5m-4)\in H \quad \text{or} \\
5n+3-(5m-4)\in H \quad \text{or} \\
5n+4-(5m-4)\in H. 
\end{cases}
\end{align*}
Considering the smallest  element in $H$ in the same congruence class modulol $5$, it is equivalent to saying that 
\begin{align*}
\begin{cases}
5n+2-(5m-4)\ge 5m+1 \quad \text{or} \\
5n+3-(5m-4)\ge 5n+7 \quad \text{or} \\
5n+4-(5m-4)\ge 5n+8
\end{cases}
\end{align*}
 Of these three inequalities, only  the first can happen. Hence, we get $2m-1\le n$. Since $5n+7= 2(5m+1)+ 5(n-2m+1)$ we obtain that  
\[
H= \langle 5, 5m+1, 5n+8, 5n+9 \rangle.
\]

Clearly, $5n+2 \notin H$ if and only if $5n+2\notin \langle 5, 5m+1\rangle$. On the other hand, for any integer $p$, $5p+2 \in \langle 5, 5m+1\rangle$ if and only if $5p+2=5u+(5m+1)v$ for some nonnegative integers $u,v$, equivalently, $5p+2 \geq  2(5m+2)$, i.e. $p\geq 2m$. Thus $n\leq 2m-1$, which forces $n=2m-1$ and 
$$
H=\langle 5, 5m+1,  10m+3, 10m+4\rangle \text{ with } m \geq 1.
$$ 
Conversely, when $H$ is of this form, it is easy to check that ${\rm PF}(H)=\{10m-3, 10m-2, 10m-1\}$, hence $R$ is far-flung Gorenstein.

Next let us assume that $C=\langle 1, t, t^3 \rangle$. Then, since $\rme(R)=5$, we have  either
\begin{align*}
\text{{\rm (2-1):}}& \quad \mathrm{PF}(H)=\{ 5n+1, 5n+3, 5n+4\} \quad \text{or}\\
\text{{\rm (2-2):}}& \quad \mathrm{PF}(H)=\{ 5n+4, 5n+6, 5n+7\}
\end{align*}
for $n\ge 0$. 

{\rm (2-1):} In this case we can choose $m>0$ so that $H=\langle 5, 5n+6, 5m+2, 5n+8, 5n+9 \rangle$ and $5m+2-5\not\in H$. Since $5m-3\not\in \mathrm{PF}(H)$, we get 
\begin{align*}
\begin{cases}
5n+1-(5m-3)\in H \quad \text{or} \\
5n+3-(5m-3)\in H \quad \text{or} \\
5n+4-(5m-3)\in H. 
\end{cases}
\end{align*}
Considering the smallest element in $H$  with the same residue modulo $5$, it is equivalent to saying that 
\begin{align*}
\begin{cases}
5(n-m)+4\ge 5n+9 \quad \text{or} \\
5(n-m)+6\ge 5n+6 \quad \text{or} \\
5(n-m)+7\ge 5m+2
\end{cases}
\end{align*}
 Of these, only the latter inequality  can happen.  Hence, we get $2m-1\le n$.  Then $5n+9=5(n-2m+1)+2(5m+2)$ and  
\[
H=\langle 5, 5m+2, 5n+6,  5n+8\rangle.
\]

The fact that $5n+4\notin H$ is equivalent to $5n+4\notin \langle 5, 5m+2\rangle$. For any integer $p$, $5p+4 \in \langle 5, 5m+2\rangle$ if and only if $5p+4 \geq 2(5m+2)$, i.e. $p\geq 2m$. We derive $n\leq 2m-1$, hence $n=2m-1$ and 
$$
H=\langle 5, 5m+2, 10m+1, 10m+3\rangle, \text{ where } m\geq 1.
$$
It is routine to check that such a semigroup has the   pseudo-Frobenius numbers as in (2-1).

{\rm (2-2):} In this case we find  $m>0$ so that  $H=\langle 5, 5n+11, 5n+12, 5m+3, 5n+9 \rangle$ and $(5m+3)-5\not\in H$. Since $5m-2\not\in \mathrm{PF}(H)$, we get 
\begin{align*}
\begin{cases}
5n+4-(5m-2)\in H \quad \text{or} \\
5n+6-(5m-2)\in H \quad \text{or} \\
5n+7-(5m-2)\in H. 
\end{cases}
\end{align*}
Arguing as before, it is equivalent to saying that 
\begin{align*}
\begin{cases}
5(n-m)+6\ge 5n+11 \quad \text{or} \\
5(n-m)+8\ge 5m+3 \quad \text{or} \\
5(n-m)+9\ge 5n+9.
\end{cases}
\end{align*}
Of these three inequalities, only  the middle one can occur, hence  $2m-1\le n$. We write $5n+11=2(5m+3)+5(n-2m+1)$ and thus 
\[
H= \langle 5,5m+3,  5n+9, 5n+12 \rangle.
\]
Since $5n+6\notin H$, arguing as before we see that $n\leq 2m-1$, and in fact $n=2m-1$. Hence
$$
H=\langle 5, 5m+3, 10m+4, 10m+7\rangle, \text{ where } m\geq 1.
$$
Conversely, it is easy to check that $10m-1, 10m+1, 10m+2$ are indeed the pseudo-Frobenius numbers of  such an $H$.
\end{proof}



\end{document}